\documentclass[mathpazo]{cicp}

\usepackage{bm}
\usepackage{mathdots}
\usepackage{extarrows}
\allowdisplaybreaks
\usepackage[top=1in,bottom=1in,left=1.25in,right=1.25in]{geometry}
\usepackage{amsfonts}
\usepackage{amsthm}
\usepackage{amssymb,indentfirst}
\usepackage{amsmath}
\usepackage[divps]{xcolor}

\usepackage{mathrsfs}
\usepackage{cases}
\usepackage{titletoc}
\usepackage{multicol}

\usepackage{amssymb}
\newtheorem{thm}{Theorem}[section]
\newtheorem{cor}[thm]{Corollary}
\newtheorem{lem}[thm]{Lemma}

\newtheorem{rem}[thm]{Remark}
\newtheorem{ex}[thm]{Example}

\def\bc{\begin{center}}       \def\ec{\end{center}}
\def\be{\begin{equation}}     \def\ee{\end{equation}}
\def\ba{\begin{array}}        \def\ea{\end{array}}
\def\bea{\begin{eqnarray}}    \def\eea{\end{eqnarray}}
\def\beaa{\begin{eqnarray*}}  \def\eeaa{\end{eqnarray*}}

\def\ifl{\iffalse}

\begin{document}
\title[Series solutions to the 3D cauchy problem for some incompressible Navier-Stokes and Euler Equations]{Series solutions to the 3D cauchy problem for some incompressible Navier-Stokes and Euler Equations}


\author[T. Zhang, A. Chen, Fan Bai]{Tao Zhang$^{a}$, \quad Alatancang$^{b,}$\corrauth}

\address{ $^a$School of Mathematical Sciences of Inner Mongolia University, Hohhot, 010021, China\\
           $^b$Huhhot University for Nationalities, Hohhot, 010051, China}

\email{\tt zhangtaocx@163.com (T. Zhang), alatanca@imu.edu.cn (A. Chen).}



\begin{abstract}
We utilize undetermined coefficient method and an iterative method to construct the series solutions of the 3D Cauchy problem for a class of incompressible Navier-Stokes and Euler Equations. Then we can turn the Navier-Stokes Equations (Euler Equations) into the Cauchy problem for finitely (infinitely) many ordinary differential equations. We get the finite series solution of the Navier-Stokes Equations. By using some combinatorial identities techniques, we prove that the  sum of the solutions to these ordinary differential equations is an infinite series solution of the Euler Equations in some cases.
\end{abstract}




\maketitle


\section{Introduction}\label{intr}

The 3-D incompreesible Navier-Stokes equations can be formulated as the follows:
\begin{equation}\label{iee}
     \left\{
     \begin{array}{l}
        \text{div}\ u=0,\\
        u_t-\nu\Delta u+(u \cdot \nabla)u+\nabla u=f,\\
     \end{array}
     \right.
   \end{equation}
where $x=(x_1,x_2,x_3)\in \mathbb{R}^3$, $u=(u_1(x,t),u_2(x,t),u_3(x,t))$ are the components of the three-dimensional velocity field, $p=p(x,t)$ is the pressure of the fluid at a position $x$, $f=(f_1(x,t),f_2(x,t),f_3(x,t))$ are the components of a given, externally applied force, $\nu> 0$ is the viscosity. When $\nu=0$, \eqref{iee} is the  3-D incompreesible Euler equation. The Navier-Stokes and Euler equations usually  describe the motion of a fluid in $\mathbb{R}^3$. Constructing exact solutions of them is very important to understand how the real fluid will flow. Lou et al. utilized B\"{a}cklund transformation  and Darboux transformation to obtain exact solutions for the Euler equations in the vorticity form \cite{lou1,lou2}.
By using the separation method, Makino obtained the first radial solutions to the Euler and Navier-Stokes equations in 1993 \cite{ma}. In 2011 and 2012, Yuen constructed many exact solutions for the Euler equations \cite{yu1,yu2}. In 2014, Fan and Yuen  obtained a class of nonlinear exact solutions with respect to $x$ for the Euler Equation \cite{fan}.

The series solutions of the linear partial differential equations can be obtained by using the superposition principle. However, it is hardly to obtain the exact series solution for any nonlinear partial differential equations. The purpose of this paper is to construct the series solutions of some Euler and Navier-Stokes Equations. The key idea of our method mainly comes from the following properties of the series $\{e^{kx}\}_{k=0}^{+\infty}$:

(i) Eigenvalues and eigenvectors: $ \frac{\text{d}^{j}}{\text{d}x^{j}}e^{kx}=k^je^{kx},\ \ j,k=0,1,2,\cdots.$

(ii) For any $m_1,m_2=0,1,2,\cdots$, we have $e^{m_1x}e^{m_2x}=e^{(m_1+m_2)x},$ $ m_1+m_2\geq\max\{m_1,m_2\}$.

 This paper is organized as follows. In Section 2, we construct the finite series solution of some incompressible Euler and Navier-Stokes equations.
 In Section 3, we construct the former infinite series solution of some incompressible Euler equations. Moreover, by using an iterative method  with respect to (ii) and some combinatorial identities techniques, we prove that the former solution is an exactly infinite series solution of the Euler Equations in some cases.

\section{Finite series solutions of the Euler and Navier-Stokes Equations}

In this section, we construct the finite series solution of the following incompressible Euler and Navier-Stokes equations:
 \begin{numcases}{}
            u_{jt}+\sum\limits_{i=1}^{3}(u_iu_{jx_i}-\nu u_{jx_ix_i})+p_{x_j}=f_j(x,t),\ \ j=1,2,3,\label{s.1} \\
         u_{1x_1}+u_{2x_2}+u_{3x_3}=0,\ \ x=(x_1,x_2,x_3)\in \mathbb{R}^3,\ t\geq 0,\label{s.2} \\
         u_j(x,0)=\sum\limits_{k=0}^nA_{jk}\xi_{k},\ \ f_j(x,t)=\sum\limits_{k=0}^nB_{jk}(t)\xi_{k},\ \  j=1,2,3,\label{s.3}
    \end{numcases}
    where $\nu\geq 0$, $A_{jk}\in \mathbb{R}$, $B_{jk}(t)\in C[0,+\infty)$,
    $\xi_{k}=\exp(k(\lambda_1 x_1+\lambda_2 x_2+\lambda_3x_3))$, $\lambda_j\in \mathbb{R}\setminus\{0\}$, $j=1,2,3$, $k=0,1,2,\cdots,n$.

Suppose that \eqref{s.1}-\eqref{s.3}  has a solution in the form:
 \begin{equation}\label{s.000}
    \left\{
    \begin{array}{l}
      u_j(x,t)=\sum\limits_{k=0}^nT_{jk}(t)\xi_k,\ \ j=1,2,3;\\
      p(x,t)=\sum\limits_{k=0}^nT_{4k}(t)\xi_k,
    \end{array}
    \right.
    \end{equation}
   where $T_{jk}(t)$, $j=1,2,3,4,$ $k=0,1,\cdots,n$ are functions to be determined.
Then we have:
   \begin{rem}\label{rem22}
   By \eqref{s.2} we have
   \begin{equation*}
    \sum\limits_{j=1}^3\sum\limits_{k=1}^n\lambda_jk T_{jk}\xi_{k}=\sum\limits_{k=1}^n\sum\limits_{j=1}^3\lambda_j kT_{jk}\xi_{k}=0,
   \end{equation*}
   Note that the sequence $\{\xi_{k}\}_{k=1}^n$ is linearly independent, so we get
   \begin{equation*}
    \sum\limits_{j=1}^3\lambda_j T_{jk}=0,\ \ k=1,2,\cdots,n.
   \end{equation*}
  Hence
   \begin{equation*}
    \sum\limits_{i=1}^{3}u_iu_{jx_i}= \sum\limits_{i=1}^3 \sum\limits_{k=0}^n \sum\limits_{m=1}^n\lambda_imT_{ik}T_{jm}\xi_{k+m}
    =  \sum\limits_{k=0}^n \sum\limits_{m=1}^n\sum\limits_{i=1}^3\lambda_iT_{ik}mT_{jm}\xi_{k+m}=0,\ \ j=1,2,3.
   \end{equation*}
   \end{rem}

Next we seek $T_{jk}(t)$, $j=1,2,3,4,$ $k=0,1,\cdots,n$.
Substituting \eqref{s.000} into \eqref{s.1}-\eqref{s.3} we get
\begin{equation*}
  \left\{
  \begin{array}{c}
     T'_{j0}-B_{j0}(t)+\sum\limits_{k=1}^{n}\left[T'_{jk}+\sum\limits_{i=1}^{3}\nu \lambda_i^2k^2T_{jk}+\lambda_jkT_{4k}-B_{jk}(t)\right]\xi_k=0,\ \ j=1,2,3, \\
     \sum\limits_{k=1}^{n}(\lambda_1T_{1k}+\lambda_2T_{2k}+\lambda_3T_{3k})k\xi_k=0, \\
         u_j(x,0)=\sum\limits_{k=0}^nT_{jk}(0)\xi_{k}=\sum\limits_{k=0}^nA_{jk}\xi_{k},\ \  j=1,2,3,
  \end{array}
  \right.
\end{equation*}
Note that the sequence $\{\xi_{k}\}_{k=1}^n$ is linearly independent, so we have
\begin{equation}\label{s.0001}
  \left\{
  \begin{array}{c}
     T'_{j0}-B_{j0}(t)=0, \\
     T_{j0}(0)=A_{j0},
  \end{array}
  \ \  j=1,2,3,
  \right.
\end{equation}
\begin{equation}\label{s.0002}
  \left\{
  \begin{array}{c}
    T'_{jk}+\sum\limits_{i=1}^{3}\nu \lambda_i^2k^2T_{jk}+\lambda_jkT_{4k}-B_{jk}(t)=0,\ \ j=1,2,3, \\
     \lambda_1T_{1k}+\lambda_2T_{2k}+\lambda_3T_{3k}=0, \\
   T_{jk}(0)=A_{jk},\ \  j=1,2,3,
  \end{array}
  \right.
\end{equation}
where $k=1,\cdots,n$.
 For every $k=1,\cdots,n$, the first equation in \eqref{s.0002} is multiplied by $\lambda_j$ (j=1,2,3). Together with
 \begin{equation*}
   \lambda_1T'_{1k}+\lambda_2T'_{2k}+\lambda_3T'_{3k}=0,
 \end{equation*}
 then we can induce that
 \begin{equation*}
  \sum\limits_{j=1}^{3}\lambda^2_jkT_{4k}-\sum\limits_{j=1}^{3}\lambda_jB_{jk}(t)=0.
 \end{equation*}
Finally we obtain
    \begin{equation*}
    \left\{
    \begin{array}{l}
      T_{j0}(t)=\int_0^tB_{j,0}(s)\text{d}s+A_{j,0},\ \ \ \ \ \  j=1,2,3,\\
     T_{40}=1,\\
    T_{4k}(t)=\frac{\sum_{j=1}^3\lambda_jB_{jk}(t)}{\sum_{j=1}^3\lambda_j^2k},\ \ \ \ \ \ k=1,2,\cdots,n,\\
    T_{jk}(t)=\exp(M_{k}(t))(\int_{0}^{t}(B_{jk}-\lambda_jk T_{4k})\exp(-M_{k}(s))\text{d}s+A_{jk}),\ \
   j=1,2,3, \ \ k=1,2,\cdots,n,\\
    \end{array}
    \right.
    \end{equation*}
   where
   \begin{equation*}
     M_{k}(t)=\sum\limits_{i=1}^{3}\nu (\lambda_ik)^2t-\int_0^t\lambda_ik T_{i0}(s)\text{d}s,\ \ \ \ k=1,2,\cdots,n.
   \end{equation*}

\section{Infinite series solution of the Euler Equations}
In this section,
 we consider the Cauchy problems for the following 3D incompressible Euler Equation:
\begin{equation}\label{sss}
     \left\{
     \begin{array}{l}
        u_{jt}+\sum\limits_{i=1}^{3}u_iu_{jx_i}+p_{x_j}=0,\ \ j=1,2,3,\\
         u_{1x_1}+u_{2x_2}+u_{3x_3}=0,\ \ x=(x_1,x_2,x_3)\in\mathbb{R}_+^3,\ t\geq 0,\\
         u_j(x,0)=\sum\limits_{k\in \mathbb{N}^3}B_{jk}\varphi_{k}\in C^{\infty}(\mathbb{R}_+^3),\ \  j=1,2,3.
     \end{array}
     \right.
   \end{equation}
where $\mathbb{R}_+^{3}=\{(r_1,r_2,r_3)\in \mathbb{R}^{3}\mid r_j\geq 0,\ j=1,2,3\}$, $\varphi_{k}=\exp(-k_1 x_1-k_2 x_2-k_3x_3),\ \ k=(k_1,k_2,k_3)\in \mathbb{N}^3=\{(k_{1},k_2,k_3)\mid k_{j}=0,1,2,\cdots,\ j=1,2,3\}.$

There are a lot of functions can be approximated by the series $\{\varphi_{k}\}_{k\in \mathbb{N}^3}$. Next we let
$$TE(U)=\left\{f(x)\in C^{\infty}(U)\mid f(x)=
       \sum\limits_{k\in \mathbb{N}^{3}} A_{k}\varphi_{k},\ \{A_k\}_{k\in \mathbb{N}^{3}}\subseteq\mathbb{R},\
       x=(x_1,x_2,x_3)\in U\right\}$$
where $U\subseteq \mathbb{R}^3$. Then by Stone-Weierstrass theorem\cite{33}, we have:

       \begin{thm} If $U\subseteq \mathbb{R}^3$ is a bounded closed set, then
      for any $f(x)\in C(U)$, there exists a sequence
     $\{f_{m}(x)\}_{m\in \mathbb{N}}\subseteq TE(U)$ such that
     \begin{equation*}
        \lim_{m\rightarrow +\infty}\ \sup\limits_{x\in U}|f_{m}(x)-f(x)|=0.
     \end{equation*}
     \end{thm}

Moreover, there exists many functions can be expressed by the series $\{\varphi_{k}\}_{k\in \mathbb{N}^3}$.
         \begin{thm} Let $f(x_1,x_2,x_3)\in C^{\infty}(U)$, $j=1,2,3,$ $U\subseteq\mathbb{R}^3$, and let $g(x_1,x_2,x_3)=f(e^{-x_1},e^{-x_2},$ $e^{-x_3})$, then we have $g(x_1,x_2,x_3)\in TE(U_0)$, $U_0=\{(x_1,x_2,x_3)\mid (e^{-x_1},e^{-x_2},e^{-x_3})\in U\}$.
        \end{thm}

       For example:
          \begin{ex}
          \begin{equation*}
           \frac{\sin e^{-x_1} \cos e^{-x_2}}{1-e^{-x_3}}=\sum\limits_{k_1=1}^{+\infty}\sum\limits_{k_2=0}^{+\infty}\sum\limits_{k_3=0}^{+\infty}
           (-1)^{k_1+k_2-1}\frac{\varphi_{(2k_1-1,2k_2,k_3)}}{(2k_1-1)!(2k_2)!}\in TE(\mathbb{R}^2\oplus (0,+\infty)).
          \end{equation*}
          \end{ex}

 Next we seek the following formal series solution of \eqref{sss} :
  \begin{equation}\label{s.4}
    \left\{
    \begin{array}{l}
      u_j(x,t)=\sum\limits_{k\in \mathbb{N}^3}T_{jk}(t)\varphi_k,\ \ j=1,2,3;\\
      p(x,t)=\sum\limits_{k\in \mathbb{N}^3}T_{4k}(t)\varphi_k,
    \end{array}
    \right.
    \end{equation}
   where $T_{jk}(t)$, $j=1,2,3,4,\ k\in \mathbb{N}^3$ are functions to be determined.
    Then substituting \eqref{s.4} into \eqref{sss} we get
    \begin{equation*}
    \left\{
      \begin{array}{c}
         T'_{m,(0,0,0)}+ \sum\limits_{k> (0,0,0)}\left[T'_{mk}-
         \sum\limits_{j=1}^{3}\sum\limits_{k^{[1]}+k^{[2]}=k}k^{[2]}_{j}T_{jk^{[1]}}T_{mk^{[2]}}
    -k_m T_{4k}\right]\varphi_{k}=0,\ \
     m=1,2,3,\\
     -\sum\limits_{k\in \mathbb{N}^3}(k_1 T_{1k}+k_2 T_{2k}+k_3 T_{3k})\varphi_{k}=0,\\
     u_j(x,0)=\sum\limits_{k\in \mathbb{N}^3}B_{jk}\varphi_{k}=\sum\limits_{k\in \mathbb{N}^3}T_{jk}(0)\varphi_k,\ \  j=1,2,3.
      \end{array}
      \right.
    \end{equation*}
    Clearly $\{T_{jk}(t)\mid j=1,2,3,4,\ k\in \mathbb{N}^3\}$ should satisfy:
     \begin{numcases}{}
         u_j=\sum\limits_{k\in \mathbb{N}^3}T_{jk}(t)\varphi_k\in C(\mathbb{R}_+^3\oplus[0,+\infty)),\ \ \ \ j=1,2,3, \label{11nb}\\
        p=\sum\limits_{k\in \mathbb{N}^3}T_{4k}(t)\varphi_{k}\in C(\mathbb{R}_+^3\oplus[0,+\infty)), \label{22nb}\\
        u_{jt}=\sum\limits_{k\in \mathbb{N}^3}T'_{jk}(t)\varphi_{k}\in C(\mathbb{R}_+^3\oplus[0,+\infty)),\  \ \ \ j=1,2,3, \label{33nb}\\
        u_{jx_i}=\sum\limits_{k\in \mathbb{N}^3} -k_iT_{jk}(t)\varphi_{k}\in C(\mathbb{R}_+^3\oplus[0,+\infty)),\  \ \ \ i,j=1,2,3, \label{44nb}\\
         p_{x_j}=\sum\limits_{k\in \mathbb{N}^3}-k_jT_{4k}(t)\varphi_{k}\in C(\mathbb{R}_+^3\oplus[0,+\infty)),\  \ \ \ j=1,2,3, \label{66nb}\\
         u_iu_{jx_i}=\sum\limits_{k\in \mathbb{N}^3}\eta_{ijk}\varphi_{k} \in C(\mathbb{R}_+^3\oplus[0,+\infty)),\ \ \ \  i,j=1,2,3, \label{77nb}
    \end{numcases}
    where
    \begin{equation*}
     \eta_{ijk}=\sum\limits_{k^{[1]}+k^{[2]}=k}-k^{[2]}_{i}
     T_{ik^{[1]}}T_{jk^{[2]}},\ \
      k^{[m]}=(k_{1}^{[m]},k_{2}^{[m]},k_{3}^{[m]})\in \mathbb{N}^3,\ \
      m=1,2,\ i,j=1,2,3.
    \end{equation*}
    Note that the sequence $\{\varphi_{k}\}_{k\in \mathbb{N}^3}$ is linearly independent, so we have:
    \begin{equation}\label{sim1}
      \left\{
      \begin{array}{c}
          T'_{j,(0,0,0)}=0,\\
          T_{j,(0,0,0)}(0)=B_{j,(0,0,0)},
      \end{array}
      \right.\ \   j=1,2,3,
    \end{equation}
    and
    \begin{equation}\label{sim2}
      \left\{
      \begin{array}{c}
         T'_{ik}-\sum\limits_{j=1}^{3}k_jT_{j,(0,0,0)}T_{ik}- \sum\limits_{j=1}^{3}\ \sum\limits_{k^{[1]}+k^{[2]}=k,\atop k^{[1]},k^{[2]}>(0,0,0)}k^{[2]}_{j}T_{jk^{[1]}}T_{ik^{[2]}}
         -k_i T_{4k}=0,\ \ i=1,2,3,\\
         k_1 T_{1k}+k_2 T_{2k}+k_3 T_{3k}=0,\\
          T_{jk}(0)=B_{jk},\ \ \ \ \ \ \ \ j=1,2,3,
      \end{array}
      \right.
    \end{equation}
    where $k>(0,0,0)$.
    For every $k>(0,0,0)$, the first equation in \eqref{sim2} is multiplied by $k_i$ (i=1,2,3). Together with
 \begin{equation*}
    k_1T'_{1k}+k_2 T'_{2k}+k_3 T'_{3k}=0,\ \ \ \ \ \  k>(0,0,0),
 \end{equation*}
 then we can induce that
    \begin{equation*}
      \sum\limits_{i=1}^{3}k_{i}\sum\limits_{j=1}^{3}\sum\limits_{k^{[1]}+k^{[2]}=k,\atop k^{[1]},k^{[2]}>(0,0,0)}
    k^{[2]}_{j}T_{jk^{[1]}}T_{ik^{[2]}}+T_{4k}\sum\limits_{i=1}^{3}k_i^2
    =0,\ \  k>(0,0,0).
    \end{equation*}
    (where $k^{[1]},k^{[2]}<k$) It follows that
    \begin{equation*}
      \sum\limits_{i,j=1}^{3}\sum\limits_{k^{[1]}+k^{[2]}=k,\atop k^{[1]},k^{[2]}>(0,0,0)}
    k^{[1]}_{i}k^{[2]}_{j}T_{jk^{[1]}}T_{ik^{[2]}}+T_{4k}\sum\limits_{i=1}^{3}k_i^2
    =0,\ \  k>(0,0,0).
    \end{equation*}
   Finally we use an simple iterative method to get $T_{jk}(t)$, $j=1,2,3,4,\ k\in \mathbb{N}^3$ by the following order:

    (i) We can get the $T_{j,(0,0,0)}(t)$, $j=1,2,3$ by solving \eqref{sim1}. Moreover, we let $ T_{4,(0,0,0)}=1$.

    (ii) Base on (i), we solve \eqref{sim2} when $|k|=1$, then we can get $T_{jk}(t)$, $|k|=1$, $j=1,2,3,4$.

    (iii) Base on (i) and (ii), we solve \eqref{sim2} when $|k|=2$, then we can get $T_{jk}(t)$, $|k|=2$, $j=1,2,3,4$.

    Where $|k|=k_1+k_2+k_3$. By repeating this process, we can get:
   \begin{equation*}
    \left\{
    \begin{array}{l}
      T_{j,(0,0,0)}(t)=B_{j,(0,0,0)},\ \ \ \ \ \ \ \ j=1,2,3,\\
     T_{4,(0,0,0)}=1,\\
    T_{4k}(t)=\frac{-\sum\limits_{i,j=1}^{3}\ \sum\limits_{k^{[1]}+k^{[2]}=k,\ k^{[1]},k^{[2]}>(0,0,0)}
    k^{[1]}_{i}k^{[2]}_{j}T_{jk^{[1]}}T_{ik^{[2]}}}{\sum\limits_{i=1}^{3}k_i^2}, \ \ \ \ k>(0,0,0),\\
    T_{jk}(t)=\exp(-P_{k}(t))\left(\int_{0}^{t}Q_{jk}(s)\exp(P_{k}(s))\text{d}s+B_{jk}\right),\ \
   j=1,2,3, \ \ k>(0,0,0),
    \end{array}
    \right.
    \end{equation*}
    where
    \begin{equation*}
      \left\{
      \begin{array}{l}
      Q_{ik}=\sum\limits_{j=1}^{3}\ \sum\limits_{k^{[1]}+k^{[2]}=k,\ k^{[1]},k^{[2]}>(0,0,0)}
     k^{[2]}_{j}T_{jk^{[1]}}T_{ik^{[2]}}+k_i T_{4k},\ \ \ \ i=1,2,3,\\
     P_{k}(t)=\sum\limits_{j=1}^{3}-B_{j,(0,0,0)}k_jt.
      \end{array}
      \right.
    \end{equation*}\vskip8pt

     Clearly we have:

     \begin{thm}\label{thm1}   If the series \eqref{s.4} we obtain satisfies
     \eqref{11nb}-\eqref{77nb}, then it is a solution of \eqref{sss}.
     \end{thm}

     Next we prove that the series \eqref{s.4} we obtain satisfies
     \eqref{11nb}-\eqref{77nb} in some cases. First we give some Lemmas.

    \begin{lem}\label{lem1} (Abel identities \cite{ci}) For any $n=1,2,\cdots$, we have
      \begin{itemize}
            \item [{\rm (i)}]
       $\sum\limits_{k=0}^{n}\frac{n!}{k!(n-k)!}(x+k)^{k-1}(y+n-k)^{n-k}=x^{-1}(x+y+n)^n,$

            \item [{\rm (ii)}] $\sum\limits_{k=0}^{n}\frac{n!}{k!(n-k)!}(x+k)^{k-1}(y+n-k)^{n-k-1}=(x^{-1}+y^{-1})(x+y+n)^{n-1}.$
             \end{itemize}
             \end{lem}

      \begin{cor}\label{cor2}  For any $k=1,2,\cdots$, we have
      \begin{itemize}
            \item [{\rm (i)}]
       $\sum\limits_{m=1}^{k}\frac{(k+1)!}{m!(k+1-m)!}
                            m^m(k+1-m)^{k-m}=k(k+1)^k,$

            \item [{\rm (ii)}] $\sum\limits_{m=1}^{k}\frac{(k+1)!}{m!(k+1-m)!}
m^{m-1}(k+1-m)^{k-m}=2k(k+1)^{k-1}\leq 2(k+1)^{k}.$
             \end{itemize}
             \end{cor}

       \begin{proof}
      By Lemma \ref{lem1} (i) we can induce that
      \begin{align*}
         &  \sum\limits_{m=0}^{k+1}\frac{(k+1)!}{m!(k+1-m)!}(x+m)^{m-1}(y+k+1-m)^{k+1-m}\\
        = & \sum\limits_{m=1}^{k}\frac{(k+1)!}{m!(k+1-m)!}(x+m)^{m-1}(y+k+1-m)^{k+1-m}+x^{-1}(y+k+1)^{k+1}+(x+k+1)^k\\
        =& x^{-1}(x+y+k+1)^{k+1}\\
        =& x^{-1}(y+k+1)^{k+1}+(k+1)(y+k+1)^{k}+\sum\limits_{m=2}^{k+1}\frac{(k+1)!}{m!(k+1-m)!}x^{m-1}(y+k+1)^{k+1-m}.
      \end{align*}
Then we have
\begin{align*}
   & \sum\limits_{m=1}^{k}\frac{(k+1)!}{m!(k+1-m)!}(x+m)^{m-1}(y+k+1-m)^{k+1-m}+(x+k+1)^k\\
  = & (k+1)(y+k+1)^{k}+\sum\limits_{m=2}^{k+1}\frac{(k+1)!}{m!(k+1-m)!}x^{m-1}(y+k+1)^{k+1-m}.
\end{align*}
Let x=y=0, we get
\begin{equation*}
  \sum\limits_{m=1}^{k}\frac{(k+1)!}{m!(k+1-m)!}m^{m-1}(k+1-m)^{k+1-m}=\sum\limits_{m=1}^{k}\frac{(k+1)!}{m!(k+1-m)!}m^{m}(k+1-m)^{k-m}=k(k+1)^{k}.
\end{equation*}
Similarly, we can get (ii).
      \end{proof}

       \begin{cor}\label{cor3} For any $ k_j=1,2,\cdots,\ j=1,2,3$, we have
      \begin{equation*}
\sum\limits_{1\leq m_i\leq k_i,\ i=1,2,3}\ m_1\prod\limits_{j=1}^3 \frac{m_j^{m_j-1}(k_j+1-m_j)^{k_j-m_j}}{m_j!(k_j+1-m_j)!}\leq 4k_1\prod\limits_{j=1}^n\frac{(k_j+1)^{k_j}}{(k_j+1)!}.
      \end{equation*}
      \end{cor}

     \begin{lem}  If $B_{j,(0,0,0)}\leq -2$, $j=1,2,3$, $\epsilon>0$, and if
       \begin{equation*}
         |B_{jk}|\leq  \frac{e^{-|k|(1+\epsilon)}}{10^3}\prod\limits_{j=1,2,3,\ k_j>0}\frac{k_j^{k_j-1}}{k_j!},\ \ \ \ \ \ \ k>(0,0,0),\ j=1,2,3.
       \end{equation*}
      Then we have
     \begin{equation}\label{guina1}
     |T_{ik}(t)|\leq \frac{1}{100}\prod\limits_{j=1,2,3,\ k_j>0}\frac{k_j^{k_j-1}}{k_j!}\exp\left(-\frac{1}{2}P_{k}(t)-|k|(1+\epsilon)\right),\ \ \ \ k>(0,0,0),\  i=1,2,3.
     \end{equation}
      \end{lem}

    \begin{proof} We prove \eqref{guina1} by the induction method.
    By a simple calculate we can induce that \eqref{guina1} holds when $|k|=1,2$.
    Suppose that it holds for any $|k|<n$ $(n>2)$, then by Corollary \ref{cor2} and \ref{cor3}, for any $k=(k_1,k_2,k_3)\geq(1,1,1)$, $|k|=n$(without loss of generality we suppose that $k_1\geq k_2,k_3$), we have
    \begin{align*}
      |T_{4k}(t)| & \leq\frac{\sum\limits_{m,j=1}^{3}\ \sum\limits_{k^{[1]}+k^{[2]}=k,\ k^{[1]},k^{[2]}>(0,0,0)}
     k^{[1]}_{m}k^{[2]}_{j}\left|T_{jk^{[1]}}\right|\left|T_{mk^{[2]}}\right|}{\sum\limits_{m=1}^{3}k_m^2} \\
       & \leq\frac{3}{10^4}
     \sum\limits_{(0,0,0)< (m_1,m_2,m_3)\leq k}\ \prod\limits_{j=1}^3 \frac{m_j^{m_j-1}(k_j-m_j)^{k_j-m_j-1}}{m_j!(k_j-m_j)!}\exp(-\frac{1}{2}P_{k}(t)-|k|(1+\epsilon))\\
     &\leq\frac{60}{10^4}\ \prod\limits_{j=1}^3\frac{k_j^{k_j-1}}{k_j!}\exp\left(-\frac{1}{2}P_{k}(t)-|k|(1+\epsilon)\right),\\
      |Q_{jk}(s)| & \leq\sum\limits_{j=1}^{3}\ \sum\limits_{k^{[1]}+k^{[2]}=k,\ k^{[1]},k^{[2]}>(0,0,0)}
     k^{[2]}_{j}|T_{jk^{[1]}}(s)||T_{mk^{[2]}}(s)|+k_j |T_{4k}(s)| \\
       & \leq\frac{90k_1}{10^4}\prod\limits_{j=1}^3\frac{k_j^{k_j-1}}{k_j!}\exp\left(-\frac{1}{2}P_{k}(t)-|k|(1+\epsilon)\right),\ \ \ \ j=1,2,3,\\
       |T_{jk}(t)|&\leq\exp(-P_k(t))\left(\int_{0}^{t}|Q_{jk}(s)|\exp(P_k(s))\text{d}s+|B_{jk}|\right)\\
     &\leq\frac{1}{100}\prod\limits_{j=1}^3\frac{k_j^{k_j-1}}{k_j!}\exp\left(-\frac{1}{2}P_{k}(t)-|k|(1+\epsilon)\right),\ \ \ \ j=1,2,3.
    \end{align*}
    In a similar way, we can prove that \eqref{guina1} holds for any $k=(k_1,k_2,k_3)\in \mathbb{N}^3$ with $k_1=0$ or $k_2=0$ or $k_3=0$.
\end{proof}

\begin{thm} If $B_{j,(0,0,0)}\leq -2$, $j=1,2,3$, $\epsilon>0$, and if
       \begin{equation*}
         |B_{jk}|\leq  \frac{e^{-|k|(1+\epsilon)}}{10^3}\prod\limits_{j=1,2,3,\ k_j>0}\frac{k_j^{k_j-1}}{k_j!},\ \ \ \ \ \ \ k>(0,0,0),\ j=1,2,3.
       \end{equation*}
   Then there exists a series solution of \eqref{sss} satisfying \eqref{11nb}- \eqref{77nb}.
\end{thm}

\begin{proof} We only need to prove that the series \eqref{s.4} we obtain satisfies \eqref{11nb}- \eqref{77nb}.
 Note that $\frac{n^m}{m!}\leq e^{n},\ m,n=1,2,\cdots$, so we have
\begin{equation*}
\left\{
\begin{array}{l}
  |T_{ik}(t)\varphi_k|\leq \frac{1}{100}\exp\left(-\frac{1}{2}P_{k}(t)-|k|\epsilon\right)\varphi_k< \frac{1}{100}e^{-|k|\epsilon},\ \ \ \ k>(0,0,0),\  i=1,2,3,\\
  |T_{4k}(t)\varphi_k|\leq \frac{60}{10^4}\exp\left(-\frac{1}{2}P_{k}(t)-|k|\epsilon\right)\varphi_k< \frac{60}{10^4}e^{-|k|\epsilon},\ \ \ \  k>(0,0,0).
\end{array}
\right.
\end{equation*}
Hence the series \eqref{s.4} converges absolutely on $\mathbb{R}_+^3\oplus [0,+\infty)$.
    It means that the series \eqref{s.4} continuous on $\mathbb{R}_+^3\oplus [0,+\infty)$.
Furthermore, we can prove that
   \begin{align*}
      |T'_{ik}\varphi_k| & =\left|\sum\limits_{j=1}^{3}\ \sum\limits_{k^{[1]}+k^{[2]}=k}k^{[2]}_{j}T_{jk^{[1]}}T_{ik^{[2]}}
         +k_i T_{4k}\right||\varphi_k| \\
       & \leq \left(\left|\sum\limits_{j=1}^3B_{j,(0,0,0)}k_j\right||T_{ik}|+|Q_{ik}|\right)\varphi_k\\
       & \leq \ \left(\left|\sum\limits_{j=1}^3B_{j,(0,0,0)}k_j\right|+\frac{90|k|}{10^4}\right)\exp\left(-\frac{1}{2}P_{k}(t)-|k|\epsilon\right)\varphi_k,\quad \quad  i=1,2,3,\ k>(0,0,0),\\
       |k_j T_{4k}\varphi_k|& \leq \frac{60|k|}{10^4}\exp(-\frac{1}{2}P_{k}(t)-|k|\epsilon)\varphi_k,\quad \quad \quad  j=1,2,3,\ k>(0,0,0),\\
       |\eta_{ijk}\varphi_k|& \leq \frac{30|k|}{10^4}\exp(-\frac{1}{2}P_{k}(t)-|k|\epsilon)\varphi_k,\quad \quad \quad  i,j=1,2,3,\ k>(0,0,0).
    \end{align*}
    Similarly,  we can induce that the series \eqref{s.4} we obtain satisfies \eqref{33nb}-\eqref{77nb}.
    Therefore it is a solution of \eqref{sss} by Theorem \ref{thm1}.
\end{proof}

\section*{Acknowledgments}

The paper is supported by the Natural Science Foundation of China (no. 11371185) and the Natural Science Foundation of Inner Mongolia,
China (no. 2013ZD01).


\end{document}